\documentclass[12pt,reqno]{amsart}
\usepackage{amssymb}
\usepackage{txfonts}
\usepackage{amsmath}

 \newtheorem{thm}{Theorem}[section]
 \newtheorem{cor}[thm]{Corollary}
 \newtheorem{lem}[thm]{Lemma}
 \newtheorem{prop}[thm]{Proposition}
 \newtheorem{exam}[thm]{Example}
 \theoremstyle{definition}
 \newtheorem{defn}[thm]{Definition}
 \theoremstyle{remark}
 \newtheorem{rem}[thm]{Remark}
 \numberwithin{equation}{section}

 \newcommand\ric{Ric}
 \newcommand\cur{Rm}
 \newcommand{\dist}{dist}
 
\begin{document}

\title
{Asymptotic behavior of Type III mean curvature flow on noncompact hypersurfaces}

\author{Liang Cheng, Natasa Sesum}

\dedicatory{}
\date{}


\thanks{Liang Cheng's Research partially supported by the Natural Science
Foundation of China 11201164 and a scholarship from the China Scholarship Council.
Natasa  Sesum thanks the NSF   support in DMS-0905749 and DMS-1056387. }

\address{Liang Cheng, School of Mathematics and Statistics, Huazhong Normal University,
Wuhan, 430079, P.R. CHINA}

\email{chengliang@mail.ccnu.edu.cn}

\address{Natasa Sesum: Department of Mathematics, Rutgers University, 110 Frelinghuysen road, Piscataway,
NJ 08854, USA.}

\email{natasas@math.rutgers.edu}

 \maketitle

\begin{abstract}
In this paper, we introduce the monotonicity formulas for the
mean curvature flow which are related to  self-expanders. Then we use the monotonicity to study the asymptotic behavior of  Type III  mean curvature flow on noncompact hypersurfaces.
\end{abstract}

\section{introduction}

Let $x_0:M^n\to \mathbb{R}^{n+1}$ be a complete immersed hypersurface. Consider the mean curvature
flow
\begin{equation}\label{mcf}
\frac{\partial x}{\partial t}=\vec{\mathbf{H}},
\end{equation}
with the initial data $x_0$, where $\vec{\mathbf{H}}=-H\nu$ is the mean curvature vector and $\nu$ is the outer unit
normal vector. One of the main topics of interest in the study of mean curvature flow (\ref{mcf}) is
that of singularity formation.  Since mean curvature flow always blows up at finite time on closed hypersurfaces,  the singularity formation of the mean curvature flow \eqref{mcf}  on closed hypersurfaces at the first singular time is described by Huisken \cite{H1}  as follows.
 Let $x(\cdot,t)$ be the solution to the mean curvature flow (\ref{mcf}). Let $A(\cdot,t)$ be the second fundamental form of $x(\cdot,t)$.
The solution to  mean curvature flow (\ref{mcf})  on closed hypersurfaces which blows up at finite time $T$ forms a

(1) Type I singularity if $\sup\limits_{M\times
[0,T)}(T-t)|A|^2<\infty$, \noindent

(2) Type II singularity if $\sup\limits_{M\times
[0,T)}(T-t)|A|^2=\infty$.

For noncompact hypersurfaces, solution to the mean curvature flow may exist for all times.
In \cite{EH1}, Ecker and Huisken showed that the mean curvature flow on locally Lipschitz continuous entire graph in
$\mathbb{R}^{n+1}$ exists for all time. In this paper, we study the asymptotic behavior of Type III mean curvature flow for which we have  a long time existence by definition.  In \cite{H1} Hamilton defined the Type III Ricci flow as the flow which has a long time existence and  such that $\sup_{t\in (0,\infty)} t \|\cur(g_t)\| < \infty$, where $\|\cur(g_t)\|$ is the norm of the Riemannian curvature of metric $g_t$. Analogous to the  Ricci flow we have the following definition for the mean curvature flow.
\begin{defn}\label{def_Type_III}
 We say the solution $x(\cdot,t)$ to  the mean curvature flow (\ref{mcf}) on noncompact hypersurfaces which exists for all time  forms a
 Type III singularity if $\sup\limits_{M\times
[0,+\infty)}t|A|^2<\infty$.
\end{defn}
\noindent Typical examples of Type III mean curvature flow are evolving entire graphs satisfying the linear growth condition, or equivalently
the entire graphs having the bounded gradient, which in particular implies
\begin{equation}\label{linear_growth}
V:=\langle \nu,w\rangle^{-1}\leq c,
\end{equation}
where $\nu$ is the unit normal vector of the graph and $w$ is a fixed unit vector such that $\langle \nu,w\rangle>0$.
Ecker and Huisken showed that the mean curvature flow on entire graphs satisfying the linear growth condition is  Type III (see Corollary 4.4 in \cite{EH}). Bobe \cite{B} has also related results for the cylindrical graphs.

Huisken \cite{H1} introduced his entropy  which becomes one of the most powerful tools in studying the mean curvature flow.
Recall the Huisken's entropy is defined as the integral of backward heat kernel:
\begin{equation}\label{Huisken}
\int_{M}(4\pi(T-t)) ^{-\frac{n}{2}}e^{-\frac{|x|^2}{4(T-t)}}d\mu_t.
\end{equation}
Huisken  proved his entropy (\ref{Huisken}) is monotone non-increasing in $t$ under the mean curvature flow
(\ref{mcf}). By using this monotonicity formula, Huisken  also showed that
Type I singularities of mean curvature flow are smooth asymptotically like
shrinking self-shinkers, characterized by the equation
\begin{equation}\label{shrinking_soliton}
\vec{\mathbf{H}}=-x^{\perp},
\end{equation}
where $x^{\perp}=\langle x,\nu \rangle \nu$. By using the Hamilton's Harnack estimate for the mean curvature flow \cite{Hamilton}, Huisken and Sinestrari (\cite{HS1}, \cite{HS2}) proved suitable rescaled sequence of the $n$-dimensional compact Type II mean curvature flow with positive mean curvature converges to a translating soliton  $\mathbb{R}^{n-k}\times \Sigma^k$, where $\Sigma^k$ is strictly convex.

In this paper, we study the singularity formation of the Type III mean curvature flow. For the entire graphs satisfying the linear growth condition
(\ref{linear_growth}) and in addition the estimate
\begin{equation}\label{growth_condition}
\langle x_0,\nu \rangle^2 \leq c(1+|x_0|^2)^{1-\delta}
\end{equation}
at time $t = 0$, where  $c<\infty$ and $\delta>0$, Ecker and Huisken \cite{EH}
 proved the solution to the normalized mean curvature flow
\begin{equation}\label{normalized_mcf}
\frac{\partial \overline{x}}{\partial s}=\vec{\overline{\mathbf{H}}}-\overline{x}
\end{equation}
with initial data $x_0$
converges as $s\to \infty$ to a self-expander. More precisely, in the case of entire graphs satisfying conditions (\ref{linear_growth}) and (\ref{growth_condition}), they  showed the following strong estimate
\begin{equation}
\sup\limits_{\overline{x}(M,s)}\frac{|\vec{\overline{\mathbf{H}}}+\overline{x}_s^{\perp}|^2 \overline{V}^2}{(1+\alpha|\overline{x}_s|^2)^{1-\epsilon}}\leq e^{-\beta s} \sup\limits_{x_0(M)}\frac{|\vec{\overline{\mathbf{H}}}+ x_{0}^{\perp}|^2 \overline{V}^2}{(1+\alpha| x_{0}|^2)^{1-\epsilon}},
\end{equation}
for all $\epsilon<\delta$
by applying the maximum principle under the flow (\ref{normalized_mcf}),
where $\overline{V}=\langle \overline{\nu},w\rangle^{-1}$, $\overline{\nu}$ is the unit normal vector of the graph and $w$ is a fixed unit vector such that $\langle \overline{\nu},w\rangle>0$.
In particular, this implies exponentially fast convergence on
compact subsets.

In order to study the
 singularity formation of  Type III mean curvature flow, we introduce some monotonicity formulas which are related to self-expanders.
We remark that there is a dual version of Huisken's entropy due to Ilmannen \cite{I3}:
\begin{equation}\label{dual_ver}
\frac{d}{dt}\int_{M}\rho d\mu_t=-\int_{M}|\vec{\mathbf{H}}-\frac{x^{\perp}}{2t+1}|^2\rho d\mu_t
\end{equation}
where  $\rho=(t+\frac{1}{2})^{-\frac{n}{2}}e^{\frac{|x|^2}{4(t+\frac{1}{2})}}$ and where the surfaces evolve by the mean curvature flow \eqref{mcf}.
Unfortunately, the monotonicity formula (\ref{dual_ver}) only makes sense on closed hypersurfaces.
Note that the density term $\rho d\mu_t$ is not pointwise monotone under the mean curvature flow (\ref{mcf}).
Actually, we can calculate that
$$\frac{\partial}{\partial t} \rho d\mu_t=-|\vec{\mathbf{H}}-\frac{x^{\perp}}{2t+1}|^2\rho d\mu_t -\left(\frac{n}{2t+1}+\frac{<x,\vec{\mathbf{H}}>}{2t+1}+\frac{|x^T|^2}{(2t+1)^2}\right)\rho d\mu_t.$$
If we could integrate above formula (which is for example the case when we are on compact surfaces), we would have found the second term on the right hand side is zero by the divergence theorem.

In this paper, we find that $\rho d\mu_t$ is monotone non-increasing under the following  flow, which we call {\it the drifting mean curvature flow},
\begin{equation}\label{tangent_mcf}
\frac{\partial x}{\partial t}=\vec{\mathbf{H}}+\frac{x^T}{2t+1}.
\end{equation}
 It turns out
the drifting mean curvature flow (\ref{tangent_mcf}) is equivalent to mean curvature flow (\ref{mcf}) up to tangent diffeomorphisms. We have the following result.
\begin{thm}\label{monotocity_tangent_mcf}
Let $x(\cdot ,t)$ be the solution to the drifting mean curvature flow (\ref{tangent_mcf}) with the initial data $x_0:M\to \mathbb{R}^{n+1}$ being an immersed hypersurface. Set $\rho=(t+\frac{1}{2})^{-\frac{n}{2}}e^{\frac{|x|^2}{4(t+\frac{1}{2})}}$. We have
\begin{equation}\label{monotocity_tangent_mcf_formula}
\frac{\partial}{\partial t} \rho d\mu_t=-|\vec{\mathbf{H}}-\frac{x^{\perp}}{2t+1}|^2 \rho d\mu_t.
\end{equation}
\end{thm}
Rescale the flow (\ref{tangent_mcf}) as
\begin{equation}\label{scaling}
\widetilde{x}(\cdot,s)=\frac{1}{\sqrt{2t+1}}x(\cdot,t),
\end{equation}
where $s$ is given by $s=\frac{1}{2}\log(2t+1)$. The normalized drifting mean curvature flow  then becomes
\begin{equation}\label{normalized_tangent_mcf}
\frac{\partial \widetilde{x}}{\partial s}=\vec{\widetilde{\mathbf{H}}}-\widetilde{x}^{\perp},\ \ s\geq 0.
\end{equation}
 Moreover, $(t+\frac{1}{2})^{-\frac{n}{2}}e^{\frac{|x|^2}{4(t+\frac{1}{2})}}d\mu_t$ becomes $e^{\frac{1}{2}|\widetilde{x}|^2}d\widetilde{\mu}_s$ under this rescaling. Note the stationary solutions to the normalized drifting mean
curvature flow (\ref{normalized_tangent_mcf}) are exactly self-expanders which are characterized by the
equation
\begin{equation}\label{expanding_soliton}
\vec{\mathbf{H}}=x^{\perp}.
\end{equation}
That is why we consider the normalized drifting mean curvature flow (\ref{normalized_tangent_mcf}).
An immediate corollary of Theorem \ref{monotocity_tangent_mcf} is the following monotonicity property
for the normalized drifting mean curvature flow.

\begin{cor}\label{monotocity_normalized_tangent_mcf}
Let $\widetilde{x}(\cdot ,s)$ be the solution to the normalized drifting mean curvature flow (\ref{normalized_tangent_mcf}) with the initial data $x_0:M\to\mathbb{R}^{n+1}$  being an immersed hypersurface. Set $\widetilde{\rho}=e^{\frac{1}{2}|\widetilde{x}|^2}$. We have
\begin{equation}\label{monotocity_normalized_tangent_mcf_formula}
\frac{\partial}{\partial s} (\widetilde{\rho} d\widetilde{\mu}_s)=-|\vec{\widetilde{\mathbf{H}}}-\widetilde{x}^{\perp}|^2 \widetilde{\rho} d\widetilde{\mu}_s.
\end{equation}
\end{cor}

Next we introduce a global monotonicity formula for the normalized drifting mean curvature flow (\ref{normalized_tangent_mcf}). The idea is the following: since $\frac{\partial}{\partial s} (\widetilde{\rho} d\widetilde{\mu}_s)\leq 0$ pointwise, we can choose a time-independent positive function $f_0$ such that $\int_{M}\widetilde{\rho} f_0 d\widetilde{\mu}_s$ is finite at $s=0$, and therefore $\int_{M}\widetilde{\rho} f_0 d\widetilde{\mu}_s$ is finite for all $s\geq 0$ since $\widetilde{\rho}f_0 d\widetilde{\mu}_s$ is monotone nonincreasing.

\begin{thm}\label{global_monotocity_normalized_tangent_mcf}
Let $\widetilde{x}(\cdot ,s)$ be the solution to the normalized drifting mean curvature flow (\ref{normalized_tangent_mcf}) with the initial data $x_{0}:M\to\mathbb{R}^{n+1}$  being an immersed hypersurface.
Assume that $\int_M e^{-\frac{1}{2}|x_0|^2}d \mu_0=C_0<\infty$.
Then
\begin{equation}\label{eq_a}
\int_M e^{\frac{1}{2}|\widetilde{x}|^2-|x_{0}|^2}d\widetilde{\mu}_s\leq C_0, \qquad \mbox{for all} \qquad s\ge 0,
\end{equation}
where the term
$ e^{\frac{1}{2}|\widetilde{x}|^2-|x_{0}|^2}d\widetilde{\mu}_s$ means $ e^{\frac{1}{2}|\widetilde{x}|^2(p,s)-|x_{0}(p)|^2}d\widetilde{\mu}_s(p)$ for $p\in M$ and
\begin{equation}\label{eq_b}
\int^{\infty}_{0}\int_M |\vec{\widetilde{\mathbf{H}}}-\widetilde{x}^{\perp}|^2 e^{\frac{1}{2}|\widetilde{x}|^2-|x_{0}|^2} d\widetilde{\mu}_s\leq C_0.
\end{equation}
Moreover, we have the following monotonicity formula
\begin{equation}\label{global_monotocity_normalized_tangent_mcf_formula}
\frac{d}{d s}\int_Me^{\frac{1}{2}|\widetilde{x}|^2-|x_{0}|^2}d\widetilde{\mu}_s=-\int_M|\vec{\widetilde{\mathbf{H}}}-\widetilde{x}^{\perp}|^2 e^{\frac{1}{2}|\widetilde{x}|^2-|x_{0}|^2} d\widetilde{\mu}_s.
\end{equation}
The theorem also holds when we replace the term $e^{-|x_{0}|^2}$ in (\ref{eq_a})-(\ref{global_monotocity_normalized_tangent_mcf_formula}) by a time-independent positive function $f_0$ satisfying
\begin{equation}\label{generalized_condition}
\int_M e^{\frac{1}{2}|x_0|^2}f_0d \mu_0<\infty.
\end{equation}
\end{thm}

We remark that the motivation for such monotonicity formulas originates from Perelman's work on the Ricci flow \cite{P1}.  Perelman \cite{P1} introduced the reduced volume  $\int_{M}\tau^{-\frac{n}{2}}e^{-l(y,\tau)}dvol(y)$ which is monotone non-increasing under the backward Ricci flow, where $l(y,\tau)$ is the reduced length. However the density term $\tau^{-\frac{n}{2}}e^{-l(y,\tau)}dvol(y)$ is not pointwise monotone non-increasing quantity under the backward Ricci flow.  Perelman showed that 
the density term  is monotone non-increasing under the backward Ricci flow along the $\mathcal{L}$-geodesic.
 
Feldman, Ilmanen, Ni \cite{FIN} also observed that there is a dual version of Perelman's reduced entropy  related to the Ricci flow expanders and defined by $\int_{M} t^{-\frac{n}{2}}e^{l_+(y,t)}dvol(y)$, along the forward Ricci flow. It only makes sense on closed manifolds. The first author and Zhu \cite{CZ} observed the density term $t^{-\frac{n}{2}} e^{l_+(\gamma_V(t),t)} \mathcal{L_+}J_V(t)dV^n$ is pointwise monotone non-increasing along the $\mathcal{L}_+$-geodesics under the forward Ricci flow in the similar way as in \cite{P1}. Hence one can add the weight and get that  
\[\int_{T_xM^n} t^{-\frac{n}{2}}e^{l_+(\gamma_V(t),t)}\mathcal{L_+}J_V(t)e^{-2|V|^2_{g(0)}}dx_{g(0)}(V),\]
is well-defined on noncompact manifolds and monotone non-increasing under the forward Ricci flow (see \cite{CZ}).

As an immediate application of Theorem \ref{global_monotocity_normalized_tangent_mcf}, using (\ref{eq_b}), we have the following result.
\begin{cor}\label{asymptotically_longtime}
Let  $\widetilde{x}(\cdot,s)$ be the normalized drifting mean curvature flow that exists for $s\in [0,\infty)$, with initial data $ x_{0}:M\to \mathbb{R}^{n+1}$ being an immersed hypersurface and satisfying  $\int_M e^{-\frac{1}{2}| x_{0}|^2}d \mu_0=C_0<\infty$.
Then the normalized drifting mean curvature flow (\ref{normalized_tangent_mcf}) asymptotically looks like the self-expander as time approaches infinity in the sense
$$\lim_{\tau,t\to\infty}\int_{\tau}^t \int_M |\vec{\widetilde{\mathbf{H}}}-\widetilde{x}^{\perp}|^2 e^{\frac{1}{2}|\widetilde{x}|^2-| x_{0}|^2} d\widetilde{\mu}_s\, ds = 0.$$
There exists a sequence of times $s_i\to \infty$ such that
\[\lim_{i\to\infty} \int_M |\vec{\widetilde{\mathbf{H}}}-\widetilde{x}^{\perp}|^2 e^{\frac{1}{2}|\widetilde{x}|^2-| x_{0}|^2} d\widetilde{\mu}_{s_i} = 0,\]\
where $\widetilde{x}$ stands for $\widetilde{x}(\cdot,s_i)$.
\end{cor}
\begin{rem}
If the mean curvature flow (\ref{mcf}) exists for all times, then the corresponding normalized drifting mean curvature flow (\ref{normalized_tangent_mcf}) also exists for all times.
Since the normalized drifting mean curvature flow (\ref{normalized_tangent_mcf}) is equivalent to the mean curvature flow (\ref{mcf}) up to tangent diffeomorphisms  and rescaling given by \ref{scaling}, we can view Corollary\ref{asymptotically_longtime} giving us the asymptotic behavior at infinite time for the mean curvature flow in the distribution sense.
\end{rem}

Next we use the monotonicity formula (\ref{global_monotocity_normalized_tangent_mcf_formula}) to study the asymptotic behavior of Type III mean curvature flow. We first have the following.

\begin{thm}\label{local_asymptotically}
Let $x( \cdot,t):M\to \mathbb{R}^{n+1}$ be the Type III solution to the mean curvature flow (\ref{mcf})  with initial data $x_{0}:M\to \mathbb{R}^{n+1}$ being a complete immersed hypersurface and satisfying  $\int_M e^{-\frac{1}{2}|x_{0}|^2}d\mu_{0}<\infty$. Assume that $\widetilde{x}(\cdot,s)$ is the corresponding normalized drifting mean curvature flow (\ref{normalized_tangent_mcf}) for $x( \cdot,t)$.
Denote by $ B(q,R)$ the ball in  $\mathbb{R}^{n+1}$ for some $q\in \mathbb{R}^{n+1}$ and $R>0$. If there exists an $s_N>0$ such that
$\widetilde{x}(M,s)\cap B(q,R)\neq \emptyset$ for $s>s_N$ and
\begin{equation}\label{control_condition}
|x_0(p)|\leq C_R
\end{equation}
 for any $p$ such that $ \widetilde{x}(p,s)\in B(q,R)$, where $C_R$ is a constant dependent on $R$ and independent of time $s$,
then $\widetilde{x}(M,s)\cap B(q,R)$ subconverges smoothly to the self-expander in $B(q,R)$.
\end{thm}
\begin{rem}\label{rmk}
\begin{enumerate}
\item
Let $N_s(q,R)=\widetilde{x}^{-1}(\widetilde{x}(M,s)\cap B(q,R))$.
Note that  condition (\ref{control_condition}) is needed when using the monotonicity formula (\ref{global_monotocity_normalized_tangent_mcf}) since the weighted term $e^{-| x_{0}|^2}$ may go to zero in $N_s(q,R)$ (see the proof of Theorem \ref{local_asymptotically}).

\item In view of (\ref{generalized_condition}) and the proof of Theorem \ref{local_asymptotically}, we can see that Theorem \ref{local_asymptotically} still holds if the condition (\ref{control_condition}) is generalized by
\begin{equation}\label{control_condition_2}
f_0(p)\geq c_0(R)>0
\end{equation}
 for $p \in N_s(q,R)$, where $f_0$ satisfies (\ref{generalized_condition}) in Theorem \ref{global_monotocity_normalized_tangent_mcf} and $c_0(R)$ is constant dependent on $R$ and independent of time.

 \item
 The condition (\ref{control_condition}) or (\ref{control_condition_2}) can not be removed due to an example by Huisken and Ecker (\cite{EH}) (see Remark \ref{remark}).
\end{enumerate}
\end{rem}

We will show  there exists an $R_0$ such that $\widetilde{x}(\cdot,s) \cap B(o,R_0) \neq \emptyset$ for $s$ sufficiently large.
As an immediate corollary of this fact and Theorem \ref{local_asymptotically}, we have the following.
\begin{cor}\label{asymptotically}
Let $x( \cdot,t)$ and  $\widetilde{x}(\cdot,s)$ be as in Theorem \ref{local_asymptotically}. Then if for any $R > 0$ we have
\begin{equation}\label{control_condition2}
|x_0(p)|\leq C_R,
\end{equation}
 for any $p$ such that $ \widetilde{x}(p,s)\in B(o,R)$, where $C_R$ is a constant dependent of $R$ and independent of time $s$,
then $\widetilde{x}(M,s)$ subconverges smoothly to the limiting self-expander.
\end{cor}

Next we give an application of Theorem \ref{local_asymptotically} in which  given conditions depend only on the initial time. It turns out that they imply  condition \eqref{control_condition2}.

\begin{thm}\label{asymptotically_app2}
Let $x( \cdot,t):M\to \mathbb{R}^{n+1}$ be Type III solution to the mean curvature flow (\ref{mcf})  with initial data $x_{0}:M\to \mathbb{R}^{n+1}$ being a complete immersed hypersurface satisfying $\int_M e^{-\frac{1}{2}|x_{0}|^2}d\mu_{0}<\infty$ and
\begin{equation}\label{asymptotically_app2_condition}
\mu H\geq -\langle x_0-q_0 ,\nu\rangle\geq  0
\end{equation}
 for some positive constant $\mu$ and a fixed vector $q_0$ at the initial time.
Then its corresponding normalized drifting mean curvature flow (\ref{normalized_tangent_mcf}) subconverges smoothly to the limiting  self-expander.
\end{thm}

\begin{rem}
\begin{enumerate}

\item
The two-sheeted hyperboloid of revolution  is an example satisfying Theorem \ref{asymptotically_app2}. We write the two-sheeted hyperboloid of revolution as $x_0(u,v)=(a\sqrt{u^2-1}\cos{v},a\sqrt{u^2-1}\sin{v},cu)$, $a>0$, $c>0$ and $|u|\geq 1$. Then the unit normal vector field 
\[\nu=\partial_vx_0\times \partial_ux_0=\frac{(c\sqrt{u^2-1}\cos{v},c\sqrt{u^2-1}\sin{v},-au)}{\sqrt{(a^2+c^2)u^2-c^2}},\]
the mean curvature 
\[H=\frac{c[c^2(u^2-1)+a^2(u^2+1)]}{a[(a^2+c^2)u^2-c^2]^{\frac{3}{2}}},\]
and 
\[<x_0,\nu>=-\frac{ac}{\sqrt{(a^2+c^2)u^2-c^2}}.\]
It easy see that the condition (\ref{asymptotically_app2_condition}) is satisfied for $p=0$ and $\mu$ large enough. Since each component of the 
two-sheeted hyperboloid of revolution $x_0(u,v)$ is the entire graph satisfying the linear growth condition (\ref{linear_growth}), the mean curvature flow on $x_0(u,v)$ must be Type III by the result of Ecker and Huisken. Then by Theorem \ref{asymptotically_app2}, the normalized drifting mean curvature flow (\ref{normalized_tangent_mcf}) of the two-sheeted hyperboloid of revolution subconverges smoothly to the limiting  self-expander.
Note that
$|\langle x_0,\nu\rangle|\leq C$, implying condition (\ref{growth_condition}), and hence the convergence of $M_s$ towards the limiting self-expander  was also confirmed by Ecker and Huisken  in \cite{EH}.

\item
If we assume that the Type III mean curvature flow only has  positive mean curvature, one may not have the convergence  towards a self-expander (see Example \ref{counterexample2}).
\end{enumerate}
\end{rem}

We also give another proof of Theorem \ref{asymptotically_app2}. This proof is based on following observation: Let $x( \cdot,t):M\to \mathbb{R}^{n+1}$ be solution to the mean curvature flow (\ref{mcf})  satisfying $$\mu H\geq -\langle x_0-q_0 ,\nu\rangle\geq  0$$
for some positive constant $\mu$ and a fixed vector $q_0$ at the initial time. Rescale the flow \begin{equation}\label{rescaling_constant}
x_{\mu}(\cdot,t)=\mu^{-\frac{1}{2}}(x(\cdot, \mu t)-q_0 ).
\end{equation}
Let $\widetilde{x}_{\mu}(,s)$ be the corresponding normalized drifting mean curvature flow of $x_{\mu}(\cdot,t)$. Then $\int_{M}e^{-\frac{1}{2}|\widetilde{x}_{\mu}|^2}d\widetilde{\mu}_s$ is monotone nonincreasing.
On the other hand, we also observe that if the  solution to the mean curvature flow (\ref{mcf}) has the initial data $x_{0}$ satisfying
$-\langle x_0-q_0 ,\nu\rangle\geq \mu H\geq 0$, then $\int_{M}e^{-\frac{1}{2}|\widetilde{x}_{\mu}|^2}d\widetilde{\mu}_s$ is monotone nondecreasing.

\begin{thm}\label{asymptotically_app3}
Let $x( \cdot,t):M\to \mathbb{R}^{n+1}$ be the Type III solution to the mean curvature flow (\ref{mcf}) with initial data $x_{0}:M\to \mathbb{R}^{n+1}$ being a complete immersed hypersurface satisfying $\int_M e^{-\frac{1}{2}|x_{0}|^2}d\mu_{0}<\infty$ and
$$-\langle x_0-q_0 ,\nu\rangle\geq \mu H\geq 0$$ for some positive constant $\mu$ and a fixed vector $q_0$ at the initial time.
Let $\widetilde{x}(,s)$ be the corresponding normalized drifting mean curvature flow of $x_{\mu}(\cdot,t)$ defined in (\ref{rescaling_constant}). Then $\int_{M}e^{-\frac{1}{2}|\widetilde{x}_{\mu}|^2}d\widetilde{\mu}_s$ is monotone nondecreasing in $s$. If $$\lim\limits_{s\to +\infty}\int_{M}e^{-\frac{1}{2}|\widetilde{x}_{\mu}|^2}d\widetilde{\mu}_s<\infty,$$ then its corresponding normalized drifting mean curvature flow (\ref{normalized_tangent_mcf}) subconverges smoothly to the limiting  self-expander. In particular, if \[vol({\widetilde{M}_s\cap B(o,R)})\leq CR^m,\] 
for some $m>0$, then $\lim\limits_{s\to +\infty}\int_{M}e^{-\frac{1}{2}|\widetilde{x}_{\mu}|^2}d\widetilde{\mu}_s\leq C$ and the result above holds.
\end{thm}

Recall that Ecker and Huisken  (\cite{EH}) proved that the normalized mean curvature flow of
entire graphs satisfying the linear growth condition
(\ref{linear_growth}) and the estimate
\begin{equation}\label{growth_condition_Ecker}
\langle x_0,\nu \rangle^2 \leq c(1+|x_0|^2)^{1-\delta}
\end{equation}
at time $t = 0$, where  $c<\infty$ and $\delta>0$,  converges to the self-expander.
There exists an example showing that
 the normalized mean curvature flow of
entire graphs satisfying only the linear growth condition
(\ref{linear_growth}) and failing to satisfy \eqref{growth_condition_Ecker} may not subconverge to a self-expander even if it has the positive mean curvature (see Example \ref{counterexample2}).
 As the application to Theorem  \ref{asymptotically_app3}, we have
\begin{cor}\label{cor_app}
Let $x_0$ be the
entire graph which has the nonnegative mean curvature and satisfies the linear growth condition
(\ref{linear_growth}). Moreover, assume there exists a fixed vector $q_0$ such that  $\langle x_0-q_0 ,\nu\rangle\leq C$, where  $C$ is a positive constant. Then
 normalized drifting mean curvature flow (\ref{normalized_tangent_mcf}) with initial data $x_0$ subconverges to the limiting self-expander.
\end{cor}

The structure of this paper is as follows. In section 2 we give proofs of Theorem \ref{monotocity_tangent_mcf}, Corollary \ref{monotocity_normalized_tangent_mcf}, Theorem \ref{global_monotocity_normalized_tangent_mcf} and Corollary \ref{asymptotically_longtime}. In
section 3 we give the proofs of Theorem \ref{local_asymptotically} and Theorem \ref{asymptotically}. In section \ref{sec-q_0 roofs} we give the proof of Theorem \ref{asymptotically_app2}, Theorem \ref{asymptotically_app3} and Corollary \ref{cor_app}.

\section{Monotonicity formulas}

Recall that the drifting mean curvature flow \eqref{tangent_mcf} is equivalent to \eqref{mcf} up to tangent diffeomorphisms defined by $\frac{x^T}{2t+1}$. Indeed, let $x$ solve $\frac{\partial}{\partial t} x = -H\nu$ and let $\phi_t = \phi(\cdot,t)$ be a family of diffeomorphisms on $M$ satisfying
\[2D_q \left(\frac{x}{t+\frac{1}{2}}(\phi(p,t),t\right)\left(\frac{\partial\phi}{\partial t}(p,t)\right) = \left(\frac{\partial}{\partial t}\left(\frac{x}{t+\frac{1}{2}}\right)(\phi(p,t),t)\right)^T,\]
implying
\[D_q x(\phi(p,t),t) \left(\frac{\partial \phi}{\partial t}(p,t)\right) = \frac{x(\phi(p,t),t)^T}{2t+1}.\]
Define $y(p,t) = x(\phi(p,t),t)$. Then $y(p,t)$ solves the drifting mean curvature flow equation,
\[\frac{\partial}{\partial t}y = \frac{\partial}{\partial t} x + D_q x(\phi(p,t),t) \left(\frac{\partial}{\partial t}\phi(p,t)\right) = -H\nu + \frac{y^T}{2t+1}\]
Similarly, one can easily see that reparametrizing drifting mean curvature flow \eqref{normalized_tangent_mcf} by diffeomorphisms leads to the normalized mean curvature flow \eqref{normalized_mcf}.

\begin{proof}[Proof of Theorem \ref{monotocity_tangent_mcf}]
Under the drifting mean curvature flow (\ref{tangent_mcf}), we have
\begin{align}\label{eq_1}
\frac{\partial}{\partial t}g_{ij} &=2\partial_i (\vec{\mathbf{H}}+\frac{x^T}{2t+1}) \partial_jx \nonumber\\
&=-2Hh_{ij}+\frac{1}{t+\frac{1}{2}}\partial_i(x-x^{\perp})\partial_j x\nonumber\\
&=-2Hh_{ij}+\frac{1}{t+\frac{1}{2}}g_{ij}+\frac{1}{t+\frac{1}{2}}x^{\perp}\partial_i\partial_j x \nonumber\\
&=-2Hh_{ij}+\frac{1}{t+\frac{1}{2}}g_{ij}-\frac{1}{t+\frac{1}{2}}\langle x,\nu\rangle h_{ij},
\end{align}
where we use $x^{\perp}=<x,\nu>\nu$ and $h_{ij}=-\nu\cdot \partial_i\partial_j x$.
It follows that
\begin{align}\label{eq_2}
\frac{\partial}{\partial t}d\mu_t &=(-|\vec{\mathbf{H}}|^2+\frac{n}{2t+1}+\frac{1}{2t+1}\langle x^{\perp},\vec{\mathbf{H}}\rangle)d\mu_t.
\end{align}
Recall that $\rho = (t + \frac 12)^{-n/2}\, e^{\frac{|x|^2}{4(t+\frac 12)}}$. By (\ref{eq_1}) and (\ref{eq_2}), we get that
\begin{align}
\frac{\partial}{\partial t} \rho d\mu_t&=(-\frac{n}{2t+1}-\frac{|x|^2}{(2t+1)^2}+\frac{\langle x,\frac{\partial}{\partial t}x\rangle}{2t+1})\rho d\mu_t+ \rho\frac{\partial }{\partial t} d\mu_t\nonumber\\
&=-|\vec{\mathbf{H}}-\frac{x^{\perp}}{2t+1}|^2 \rho d\mu_t\nonumber.
\end{align}
\end{proof}

\begin{proof}[Proof of Corollary \ref{monotocity_normalized_tangent_mcf}]
Using the scaling $\widetilde{x}(\cdot,s) = \frac{x(\cdot,t)}{\sqrt{2t+1}}$ along with $s = \frac 12\log(2t+1)$, and Theorem \ref{monotocity_tangent_mcf}, we get
\begin{align*}
\frac{\partial}{\partial s} \widetilde{\rho} d\widetilde{\mu}_s &= \frac{\partial}{\partial t} \left( e^{\frac{|x|^2}{4(t+\frac{1}{2})}} \frac{d\mu_t}{(2t+1)^{\frac n2}}\right) \frac{d t}{ds} \\
&= -(2t+1)\,|\vec{\mathbf{H}} - \frac{x^{\perp}}{2t+1}|^2\, \rho\, (d\mu_t 2^{-\frac n2}) \\
&= -|\vec{\widetilde{\mathbf{H}}} - \widetilde{x}^{\perp}|^2 \widetilde{\rho} d\widetilde{\mu}_s
\end{align*}
\end{proof}

Next we give the proof of Theorem \ref{global_monotocity_normalized_tangent_mcf}  whose immediate consequence is Corollary \ref{asymptotically_longtime}.

\begin{proof}[Proof of Theorem \ref{global_monotocity_normalized_tangent_mcf}] 
Since the weighted term $e^{-| x_{0}|^2}$ is independent of time,
we have
$$\frac{\partial}{\partial s} e^{\frac{1}{2}|\widetilde{x}|^2-| x_{0}|^2} d\widetilde{\mu}_s=-|\vec{\widetilde{\mathbf{H}}}-\widetilde{x}^{\perp}|^2 e^{\frac{1}{2}|\widetilde{x}|^2-| x_{0}|^2} d\widetilde{\mu}_s.$$
Integrate above over compact domain $\Omega$ in $M$,
we get
\begin{equation}\label{eq_d}
\frac{d}{d s}\int_{\Omega} e^{\frac{1}{2}|\widetilde{x}|^2-| x_{0}|^2} d\widetilde{\mu}_s=-\int_{\Omega}|\vec{\widetilde{\mathbf{H}}}-\widetilde{x}^{\perp}|^2 e^{\frac{1}{2}|\widetilde{x}|^2-| x_{0}|^2} d\widetilde{\mu}_s\leq 0.
\end{equation}
Then
\begin{equation}\label{eq_f}
\int_{\Omega} e^{\frac{1}{2}(|\widetilde{x}|^2-2| x_{0}|^2)}d\widetilde{\mu}_s\leq \int_{\Omega} e^{-\frac{1}{2}| x_{0}|^2}d \mu_0.
\end{equation}
Taking $\Omega\to M$, we conclude that (\ref{eq_a}) holds.
Integrate (\ref{eq_d}) over time interval $[0,s]$,
we get
\begin{equation}\label{eq_111}
\int_{\Omega} e^{-\frac{1}{2}| x_{0}|^2}d\mu_0-\int_{\Omega} e^{\frac{1}{2}|\widetilde{x}|^2-| x_{0}|^2} d\widetilde{\mu}_s=\int^s_{0}\int_{\Omega}|\vec{\widetilde{\mathbf{H}}}-\widetilde{x}^{\perp}|^2 e^{\frac{1}{2}|\widetilde{x}|^2-| x_{0}|^2} d\widetilde{\mu}_s.
\end{equation}
Then we have
\begin{equation}\label{eq_g}
\int^s_{0}\int_{\Omega}|\vec{\widetilde{\mathbf{H}}}-\widetilde{x}^{\perp}|^2 e^{\frac{1}{2}|\widetilde{x}|^2-| x_{0}|^2} d\widetilde{\mu}_s\leq \int_{\Omega} e^{-\frac{1}{2}| x_{0}|^2}d\mu_0\leq C_0.
\end{equation}
Taking $\Omega\to M$ in (\ref{eq_111}) and (\ref{eq_g}), we conclude that (\ref{eq_b}) and (\ref{global_monotocity_normalized_tangent_mcf_formula}) hold. Then Corollary \ref{asymptotically_longtime} follows from (\ref{eq_b}) and (\ref{global_monotocity_normalized_tangent_mcf_formula}) directly.
\end{proof}

\section{Convergence to an expander}

In this section we present the proofs of Theorem \ref{local_asymptotically} and Corollary \ref{asymptotically}. They give us sufficient conditions under which we have that the rescaled Type III  mean curvature flow converges to an expander.

\begin{proof}[Proof of Theorem \ref{local_asymptotically}]
 Let $N_s(q,R)=\widetilde{x}^{-1}(\widetilde{x}(M,s)\cap B(q,R))$.
 Using $|\widetilde{x}_0|^2 \le C_R^2$ on $N_s(q,R)$  and Corollary \ref{monotocity_normalized_tangent_mcf} we have
\begin{align*}
\mathcal{H}^n(\widetilde{x}(M,s)\cap B(q,R))&=\int_{M}\chi(N_s(q,R)) d\widetilde{\mu}_s\\
&\leq \int_{M}\chi(N_s(q,R))e^{C_R^2+\frac{1}{2}(|\widetilde{x}|^2-2| x_{0}|^2)} d\widetilde{\mu}_s\\
&\leq \int_{M}\chi(N_s(q,R))e^{C_R^2-\frac{1}{2}| x_{0}|^2} d \mu_0\\
&\leq e^{C_R^2}C_0,
\end{align*}
for $s>s_N$.

 Since the drifting mean curvature flow (\ref{tangent_mcf}) only differs from (\ref{mcf}) by the tangent diffeomorphisms, the drifting mean curvature flow (\ref{tangent_mcf}) is also  Type III.
By  rescaling (\ref{scaling}) we have $|\widetilde{A}(\cdot,s)|\leq C$ for $0< s<+\infty$, where $\widetilde{A}(\cdot,s)$ is the second fundamental form of immersion $\widetilde{x}(\cdot,s)$.  Moreover, we also have $|\widetilde{\nabla}^m \widetilde{A}(\cdot,s)|\leq C(m)$ by  Ecker and Huisken's derivative estimates for the mean curvature flow (see \cite{EH1}).
Moreover, $\widetilde{x}(M,s)\cap B(q,R)\neq \emptyset$ for $s>s_N$ by the assumption.  Therefore, by the result of Langer (\cite{J}) we conclude that
$\widetilde{x}(M,s_i)\cap B(q,R)$ (under reparametrization), subconverges smoothly along sequences $s_i\to \infty$ to a limiting immersion $\widetilde{x}_{\infty}$ in $B(q,R)$.   We have
\begin{align*}
& \int_M e^{\frac12(|\widetilde{x}|^2 - |\widetilde{x}_0 |^2)}\, d\widetilde{\mu}_t - \int_M e^{\frac12(|\widetilde{x}|^2 - |\widetilde{x}_0 |^2)}\, d\widetilde{\mu}_s \\
&= -\int_s^t\int_M e^{\frac12(|\widetilde{x}|^2 - |\widetilde{x}_0 |^2)}|\vec{\widetilde{\mathbf{H}}}-\widetilde{x}^{\perp}|^2 d\widetilde{\mu}_{\tau}\, d\tau.
\end{align*}
Since $\int_M e^{\frac12(|\widetilde{x}|^2 - |\widetilde{x}_0 |^2)}\, d\widetilde{\mu}_t$ is uniformly bounded and decreasing function in $t$, there exists a finite $\lim_{t\to\infty} \int_M e^{\frac12(|\widetilde{x}|^2 - |\widetilde{x}_0 |^2)}\, d\widetilde{\mu}_t$ implying that
\[\lim_{s\to\infty} \int_s^{\infty} \int_M e^{\frac12(|\widetilde{x}|^2 - |\widetilde{x}_0 |^2)}\, |\vec{\widetilde{\mathbf{H}}}-\widetilde{x}^{\perp}|^2\, d\widetilde{\mu}_{\tau}\, d\tau = 0.\]
Using that $|\widetilde{x}_0(p)| \le C_R$ in $N_s(q,R)$ for all $s$ sufficiently big we get
\begin{align}
\label{eq-expander1}
&0 = \lim_{s\to\infty} \int_s^{\infty}\int_M e^{\frac12(|\widetilde{x}|^2 - |\widetilde{x}_0 |^2)}\, |\vec{\widetilde{\mathbf{H}}}-\widetilde{x}^{\perp}|^2\, d\widetilde{\mu}_{\tau}\, d\tau \nonumber \\
&\ge e^{-C_R^2}\int_s^{\infty} \int_{N_s(q,R)} e^{\frac12\,|\widetilde{x}|^2}\, |\vec{\widetilde{\mathbf{H}}} - \widetilde{x}^{\perp}|^2\, d\widetilde{\mu}_{\tau}\, d\tau. \nonumber\\
&= e^{-C_R^2}\int_s^{\infty} \int_{\widetilde{x}(M,s)\cap B(q,R)} e^{\frac12\,|\widetilde{x}|^2}\, |\vec{\widetilde{\mathbf{H}}} - \widetilde{x}^{\perp}|^2\, d\widetilde{\mu}_{\tau}\, d\tau.
\end{align}
Recall that for every sequence $s_i\to\infty$, there exists a subsequence so that hypersurfaces $\widetilde{x}(M,s)\cap B(q,R)$ converge uniformly on compact sets to a limiting hypersurface in $B(q,R)$ which is defined by an immersion $\widetilde{x}_{\infty}$. Estimate \eqref{eq-expander1} implies $\widetilde{x}_{\infty}$ satisfies $\vec{\widetilde{\mathbf{H}}}_{\infty} = \widetilde{x}_{\infty}^{\perp}$ in $B(q,R)$.
\end{proof}

\begin{proof}[Proof of Corollary \ref{asymptotically}]
Let $x( \cdot,t):M\to \mathbb{R}^{n+1}$ be a Type III solution to the mean curvature flow (\ref{mcf}) with
$\sup\limits_{M\times
[0,\infty)}t|A|^2=C<\infty$ and let $\widetilde{x}(\cdot,s)$ be its corresponding normalized mean curvature flow.
By Theorem \ref{local_asymptotically}, we only need to prove there exists  $R_0$ such that $\widetilde{x}(M,s)\cap B(o,R_0)\neq \emptyset$ for $s$ sufficiently large.
Let $\overline{x}(\cdot ,s)$ be the solution to the normalized mean curvature flow
\begin{equation}\label{normalized_mcf1}
\frac{\partial \overline{x}}{\partial s}=\vec{\overline{\mathbf{H}}}-\overline{x},
\end{equation}
with the initial data $ x_{0}$. Then we have
\begin{equation}\label{eq_31}
\frac{\partial }{\partial s} |\overline{x}|^2=2\langle\vec{\overline{\mathbf{H}}},\overline{x}\rangle-2|\overline{x}|^2.
\end{equation}
Since the mean curvature flow  is  Type III and the normalized mean curvature flow (\ref{normalized_mcf1}) is obtained by
\begin{equation}
\overline{x}(\cdot,s)=\frac{1}{\sqrt{2t+1}}x(\cdot,t),
\end{equation}
where $s$ is given by $s=\frac{1}{2}\log(2t+1)$. Then $|\vec{\overline{\mathbf{H}}}|\leq C(n)$ for $[0,+\infty)$.
It follows from (\ref{eq_31}) that
$$|\overline{x}|(p,s)\leq e^{-s}| x_{0}|(p)+C(n) (1 - e^{-s}).$$
Hence $\overline{x}(M,s)\cap B(o,C(n)+1)\neq \emptyset$ for $s$  sufficiently large. Since $\bar{x}(M,s)\cap B(o,C(n)+1) \neq \emptyset$ and since the normalized drifting mean curvature flow (\ref{normalized_tangent_mcf}) differs from the normalized mean curvature flow (\ref{normalized_mcf}) only by the tangent diffeomorphisms,
(implying $\widetilde{x}(M,s)=\overline{x}(M,s)$), we have that $\widetilde{x}(M,s) \cap B(o,C(n)+1) \neq \emptyset$ for $s$  sufficiently large. By Theorem \ref{local_asymptotically}, we conclude $\widetilde{x}(M,s)\cap B(o,R)$ subconverges to the limiting self-expander in $B(o,R)$ for all $R\geq C(n)+1$.
\end{proof}

\begin{rem}\label{remark}
In \cite{EH}, Ecker and Huisken proved the following proposition showing that the normalized mean curvature flow (\ref{normalized_mcf}) on entire graphs satisfying the linear growth condition (\ref{linear_growth}) can not subconverge to a self-expander  if the condition (\ref{growth_condition}) fails.

\begin{prop}[\cite{EH}]\label{counterexample}
Let $\overline{x}:M\to\mathbb{R}^{n+1}$ be the entire graph solution to the normalized mean curvature flow
(\ref{normalized_mcf}) whose  initial data $x_{0}$ satisfies the linear growth condition (\ref{growth_condition}) and
$|\nabla^m A_0|\leq c(m)(1+|x|^2)^{-m-1}$ for $m=0,1$, where $A_0$ is the second fundamental form of $x_{0}$. Suppose there exists a sequence of points $p_k$ such that $|x_{0}(p_k)|\to \infty$
and $\langle x_{0}(p_k),\nu\rangle^2=\gamma |x_{0}(p_k)|^2$ for some $\gamma>0$. Then there exists a sequence of times $s_k\to\infty$ for which  $c_1\leq |x(p_k,s_k)|\leq c_2$ and $(H+\langle x,\nu \rangle)(p_k,s_k)$ has a uniform positive lower bound.
\end{prop}

They also gave the following explicit example which satisfies the conditions of Proposition \ref{counterexample}.
\begin{exam}\label{counterexample_1}
The graph 
\begin{equation}
u_0(\hat{x})=u_0(|\hat{x}|)=\left\{
                              \begin{array}{ll}
                                |\hat{x}|\sin\log |\hat{x}|, & \hbox{$|\hat{x}|\geq  1$;} \\
                                smooth, & \hbox{$|\hat{x}|\leq 1$,}
                              \end{array}
                            \right.
\end{equation}
where $\hat{x}$ is the coordinate on $\mathbb{R}^2$ satisfies conditions of Proposition \ref{counterexample}.
\end{exam}

It follows from Proposition \ref{counterexample} that $\overline{x}(M,s)\cap B(o,c_2)$ can not converge to the self-expander in the case of Example \ref{counterexample_1}, where $c_2$ is the same as in Proposition \ref{counterexample}.
 Since the normalized drifting mean curvature flow (\ref{normalized_tangent_mcf}) only differs from normalized mean curvature flow (\ref{normalized_mcf}) by tangent diffeomorphisms,  it follows that $\widetilde{x}(M,s)\cap B(o,c_2)$ can not converge to the self-expander in the case of Example \ref{counterexample_1}, where $\widetilde{x}(\cdot,s)$ is
 the corresponding solution to the normalized drifting flow (\ref{normalized_tangent_mcf}).
By Theorem \ref{local_asymptotically} and Remark \ref{rmk}, we know that the conditions (\ref{control_condition}) and (\ref{control_condition_2}) must fail in this case.
\end{rem}

Next we give an example which shows that one may not have that the asymptotic limit of Type III mean curvature is the self-expander even if the initial data $x_0$ is an entire graph satisfying the linear growth condition (\ref{linear_growth}) and having  the positive mean curvature.  

\begin{exam}\label{counterexample2}
Let $x_0(r,\theta)=(r\cos{\theta},r\sin{\theta},f(r))$ be a surface of revolution.
We calculate that
the first fundamental form
$$
g=(1+f'(r)^2)dr^2+r^2d\theta^2,
$$
the second fundamental form
 $$
 h=\frac{f''(r)}{\sqrt{1+f'(r)^2}}dr^2+\frac{rf'(r)}{\sqrt{1+f'(r)^2}}d\theta^2,
 $$
  mean curvature $H=\frac{f'(r)(1+f'(r)^2)+rf''(r)}{r(1+f'(r)^2)^{\frac{3}{2}}}$, $\nu=\frac{(f'(r)\sin{\theta},f'(r)\cos{\theta},-1)}{\sqrt{1+f'(r)^2}}$,
 and
   $$<x_0,\nu>=\frac{f'(r)r-f(r)}{\sqrt{1+f'(r)^2}}.$$
   Here we choose
\begin{equation}
f(r)=\left\{
                              \begin{array}{ll}
                                r\sin\log r+6r, & \hbox{$r\geq  1$;} \\
                                \text{smooth and $f'(r)\geq 0$, $f''(r)\geq 0$}, & \hbox{$r \leq 1$,}
                              \end{array}
                            \right.
\end{equation}
Such $f(r)$ exists since $f'(1)=7$ and $f''(1)=1$. It is easy to check that $f(r)$ satisfies $|f'(r)|\leq C$, $|rf''(r)|\leq C$, $|r^2f'''(r)|\leq C$
and $|f'(r_k)r_k-f(r_k)|=\gamma r_k$ for some sequence $r_k\to +\infty$ and a constant $\gamma>0$. Then it easily follows  the surface $x_0$ satisfies the conditions of Proposition \ref{counterexample} and that $H > 0$.
\end{exam}

\section{More on the convergence to an expander}
\label{sec-q_0 roofs}
In this section we give the proofs of  Theorem \ref{asymptotically_app2}, Theorem \ref{asymptotically_app3} and Corollary \ref{cor_app} where we impose conditions on the initial data and then show that condition \eqref{control_condition} is satisfied, so that the conclusion of Theorem \eqref{local_asymptotically} still holds.

We will need to apply the maximum principle for complete noncompact one parameter family of hypersurfaces that has been proved for example in \cite{EH1}. We  state this maximum principle result below for the convenience of a reader.

\begin{thm}[Maximum principle for complete manifolds in \cite{EH1}]
\label{thm-EH}
Suppose that the manifold $M^n$ with Riemannian metrics $g(\cdot,t)$ satisfies a uniform volume growth restriction, namely
\begin{equation}
\label{eq-vol}
vol_t(B_{g(t)}(p,r)) \le e^{k(1+r^2)},
\end{equation}
holds for some point $p\in M^n$ and a uniform constant $k > 0$ for all $t \in [0,T]$, where $B_{g(t)}(p,r)$ is the intrinsic ball on $M^n$. Let $f$ be a function on $M^n\times [0,T]$ which is smooth on $M^n\times (0,T]$ and continuous on $M^n\times[0,T]$. Assume that $f$ and $g(t)$ satisfy
\begin{enumerate}
\item[(i)]
$\frac{\partial}{\partial t} f \le \Delta_t f + {\bf a} \cdot \nabla f + b f$ where the function $b$ satisfies $\sup_{M^n\times [0,T]} |b| \le \alpha_0$ for some $\alpha_0 < \infty$ and the vector {\bf $a$} satisfies $\sup_{M^n\times[0,T]} |{\bf a}| \le \alpha_1$ for some $\alpha_1 < \infty$,
\item[(ii)]
$f(p,0) \le 0$ for all $p\in M^n$,
\item[(iii)]
$\int_0^T \int_M e^{-\alpha_2^2 \dist_t(p,y)^2}\, |\nabla f|^2(y,t)\, d\mu_t\, dt < \infty, \qquad \mbox{for some} \qquad \alpha_1 < \infty$,
\item[(iv)]
$\sup_{M^n\times[0,T]} \left|\frac{\partial}{\partial t} g_{ij}\right| \le \alpha_3, \qquad \mbox{for some} \qquad \alpha_3< \infty$.
\end{enumerate}
Then we have $f \le 0$ on $M^n\times [0,T]$.
\end{thm}

\begin{rem}
\label{rem-vol}
In the case the second fundamental form is at each time slice uniformly bounded in space,  since $\ric_{M_t} \ge -2|A|^2 \ge -C$, for $t\in [0,T]$ (where $C$ may depend on $T$), the uniform volume growth condition \eqref{eq-vol}  of Theorem \ref{thm-EH} holds for $t \in [0,T]$. Hence, we can apply the maximum principle for complete hypersurfaces moving by the mean curvature flow, such that  the second fundamental form is bounded at each time slice and such that (i)-(iv) of Theorem \ref{thm-EH} hold.
\end{rem}

Before presenting the proof of Theorem \ref{asymptotically_app2} we need the following lemma.

\begin{lem}\label{lemma_4.1}
Let  $\widetilde{x}(\cdot,s)$ be the solution to the normalized drifting mean curvature flow (\ref{normalized_tangent_mcf}) with the second fundamental form bounded at each time slice. Assume the initial data $x_{0}:M\to \mathbb{R}^{n+1}$ is a complete immersed hypersurface satisfying  $H+\langle x_0,\nu\rangle\geq 0$ (resp. $H+\langle x_0,\nu\rangle\leq  0$) at $s=0$. Then $\widetilde{H}+\langle\widetilde{x},\widetilde{\nu}\rangle\geq 0$ (resp.$\widetilde{H}+\langle\widetilde{x},\widetilde{\nu}\rangle\leq  0$) for all $s\geq 0$. Moreover, if $H\geq 0$ and $\langle x_0,\nu\rangle\leq 0$ at initial time, then $\widetilde{H}\geq 0$ and $\langle\widetilde{x},\widetilde{\nu}\rangle\leq  0$ for all $s\geq 0$.
\end{lem}
\begin{proof}
Let $\overline{x}(\cdot ,s)$ be the solution to the normalized mean curvature flow
(\ref{normalized_mcf1})
with the initial data $x_{0}$.
It follows from Lemma 5.5 in \cite{EH} that
\begin{equation}
(\frac{\partial}{\partial s}-\overline{\Delta})\overline{H}=|\overline{A}|^2\overline{H}+\overline{H},
\end{equation}
\begin{equation}
(\frac{\partial}{\partial s}-\overline{\Delta})\langle\overline{x},\overline{\nu}\rangle=|\overline{A}|^2\langle\overline{x},\overline{\nu}\rangle-2\overline{H}-\langle\overline{x},\overline{\nu}\rangle,
\end{equation}
\begin{equation}\label{eq4_3}
(\frac{\partial}{\partial s}-\overline{\Delta})(\overline{H}+\langle\overline{x},\overline{\nu}\rangle)=(|\overline{A}|^2-1)(\overline{H}+\langle\overline{x},\overline{\nu}\rangle),
\end{equation}
where $\overline{A}(\cdot,s)$ is the seconded fundamental form of $\overline{x}(\cdot,s)$. By Remark \ref{rem-vol} and the maximum principle for noncompact manifolds (Theorem \ref{thm-EH}), we have $\overline{H}+\langle\overline{x},\overline{\nu}\rangle\geq 0$ (resp.$\overline{H}+\langle\overline{x},\overline{\nu}\rangle\leq  0$) for all $s\geq 0$ if $H+\langle x_0,\nu\rangle\geq 0$ (resp. $H+\langle x_0,\nu\rangle\leq  0$) at $s=0$. Moreover,
$\overline{H}\geq 0$ and
$\langle\overline{x},\overline{\nu}\rangle\leq 0$  for all $s\geq 0$ if $H\geq 0$ and $\langle x_0,\nu\rangle\leq 0$ at initial time.
Since the normalized drifting flow (\ref{normalized_tangent_mcf})
only differs from (\ref{normalized_mcf1}) by the tangent diffeomorphisms, we conclude that Lemma \ref{lemma_4.1} holds.
\end{proof}

\begin{proof}[Proof of Theorem \ref{asymptotically_app2}]
The rescaled solution to the mean curvature flow $x_{\mu}(\cdot,t) := \mu^{-\frac{1}{2}}(x(\cdot, \mu t)-q_0 )$ is also the Type III solution, with the initial
data $\mu^{-\frac{1}{2}}(x_0-q_0 )$  satisfying $\langle x_{\mu}(\cdot,0), \nu\rangle = \langle\mu^{-\frac{1}{2}}(x_0-q_0 ),\nu\rangle \leq 0$, by the assumptions in Theorem \ref{asymptotically_app2}. Moreover,
\[H_{\mu}(\cdot,0) + \langle x_{\mu}(\cdot,0), \nu\rangle = \mu^{-\frac 12}\, \big(H(\cdot,0) + \langle x_0 - q_0 , \nu\rangle \big) \ge 0.\]\
Let $\widetilde{x}(\cdot,s)$ and $\widetilde{x}_{\mu}(\cdot,s)$ be the corresponding normalized drifting mean curvature flow for $x(\cdot,t)$ and $x_{\mu}(\cdot,t)$ respectively. It is easy to see that $\widetilde{x}_{\mu}(\cdot,s)$ subconverges to the limiting self-expander if and only if $\widetilde{x}(\cdot,s)$ subconverges to the limiting self-expander.
Hence, as a matter of scaling with $\mu = 1$, without losing the generality we can assume that $q_0 = 0$ and that
\[\langle x_0, \nu \rangle \le 0, \qquad \mbox{and} \qquad  H(\cdot,0)+ \langle x_0, \nu\rangle  \ge 0.\]
By Lemma \ref{lemma_4.1}, $\langle\widetilde{x},\widetilde{\nu}\rangle\leq 0$ and $\widetilde{H}+\langle\widetilde{x},\widetilde{\nu}\rangle\geq 0$ for all $s\geq 0$.

  Moreover, we have
$$
\frac{\partial}{\partial s}|\widetilde{x}|^2=2\langle \vec{\widetilde{H}}-\widetilde{x}^{\perp}, \widetilde{x}\rangle=-2\langle\widetilde{x},\widetilde{\nu}\rangle(\widetilde{H}+\langle\widetilde{x},\widetilde{\nu}\rangle)\geq 0.
$$
This implies that for every $p$ such that $\widetilde{x}(p,s) \in B(o,R)$, that is, $|\widetilde{x}(p,s)| \le R$, we have
\[|\widetilde{x}(p,0)| \le |\widetilde{x}(p,s)| \le R.\]
Then Theorem \ref{asymptotically_app2} follows from Corollary \ref{asymptotically} immediately.
\end{proof}

\begin{rem}
\label{another-proof}
We also give another proof of Theorem \ref{asymptotically_app2} without applying Corollary \ref{asymptotically}. As shown in the proof presented above
 we can assume that $p = 0$ and that
\[\langle x_0, \nu \rangle \le 0, \qquad \mbox{and} \qquad \langle x_0, \nu_0\rangle + H(\cdot,0) \ge 0,\]
without losing the generality.
 We know that $\langle\widetilde{x},\widetilde{\nu}\rangle\leq 0$ and $\widetilde{H}+\langle\widetilde{x},\widetilde{\nu}\rangle\geq 0$ for all $s\geq 0$ by Lemma \ref{lemma_4.1}. We compute that
\begin{align*}\label{eq4_1}
\frac{\partial}{\partial s}\widetilde{g}_{ij} &=2\partial_i (\vec{\widetilde{\mathbf{H}}}-\widetilde{x}^{\perp}) \partial_j\widetilde{x} \nonumber\\
&=-2\widetilde{H}\widetilde{h}_{ij}+2\widetilde{x}^{\perp}\partial_i\partial_j \widetilde{x} \nonumber\\
&=-2\widetilde{H}\widetilde{h}_{ij}-2\langle\widetilde{x},\widetilde{\nu}\rangle\widetilde{h}_{ij},
\end{align*}
where we use $\widetilde{x}^{\perp}=\langle\widetilde{x},\widetilde{\nu}\rangle\widetilde{\nu}$ and $\widetilde{h}_{ij}=-\widetilde{\nu}\cdot \partial_i\partial_j \widetilde{x}$.
It follows that
\begin{align*}
\frac{\partial}{\partial s}d\widetilde{\mu}_s &=(-|\vec{\widetilde{\mathbf{H}}}|^2+\langle\widetilde{x}^{\perp},\vec{\widetilde{\mathbf{H}}}\rangle)d\widetilde{\mu}_s\nonumber\\
&=-\widetilde{H}(\widetilde{H}+\langle\widetilde{x},\widetilde{\nu}\rangle)d\widetilde{\mu}_s.
\end{align*}
We have
$$\frac{\partial}{\partial s}|\widetilde{x}|^2=2\langle \vec{\widetilde{H}}-\widetilde{x}^{\perp}, \widetilde{x}\rangle=-2\langle\widetilde{x},\widetilde{\nu}\rangle(\widetilde{H}+\langle\widetilde{x},\widetilde{\nu}\rangle).$$
It follows that
\begin{equation}\label{monotocity_formula1}
\frac{\partial}{\partial s}e^{-\frac{1}{2}|\widetilde{x}|^2}d\widetilde{\mu}_s=\langle\widetilde{x},\widetilde{\nu}\rangle^2-\widetilde{H}^2=
(\langle\widetilde{x},\widetilde{\nu}\rangle-\widetilde{H})(\langle\widetilde{x},\widetilde{\nu}\rangle+\widetilde{H})\leq 0.
\end{equation}
Then we have
\begin{equation*}
\frac{d}{d s}\int_Me^{-\frac{1}{2}|\widetilde{x}|^2}d\widetilde{\mu}_s=
\int_M(\langle\widetilde{x},\widetilde{\nu}\rangle-\widetilde{H})(\langle\widetilde{x},\widetilde{\nu}\rangle+\widetilde{H})e^{-\frac{1}{2}|\widetilde{x}|^2}d\widetilde{\mu}_s\leq 0.
\end{equation*}
It follows that
\begin{align*}
e^{-\frac{1}{2}R^2}H^n(\widetilde{x}(M,s)\cap B(o,R))\leq\int_{\widetilde{M}_s\cap B(o,R)}e^{-\frac{1}{2}|\widetilde{x}|^2}d\widetilde{\mu}_s
\leq \int_{M}e^{-\frac{1}{2}|\widetilde{x}_0|^2}d\widetilde{\mu}_0\leq C_0.
\end{align*}
Hence
$H^n(\widetilde{x}(M,s)\cap B(o,R))\leq C_0e^{\frac{1}{2}R^2}$ for any $R>0$. Note that all derivatives of the second fundamental form of hypersurfaces $\widetilde{x}(M,s)$ are uniformly bounded, which is a consequence of our Type III assumption, rescaling \eqref{scaling} and Ecker and Huisken's gradient estimates (see \cite{EH1}).
Moreover, by the Type III assumption, similarly as in the proof of Corollary \ref{asymptotically}  we can show that $\widetilde{x}(M,s)\cap B(o,R_0))\neq \emptyset$ for some $R_0$.
As a result we conclude that
$\widetilde{x}(M,s)\cap B(q,R)$ (under reparametrization) along sequences (as $s_i\to \infty$) subconverges smoothly to a limiting immersion $\widetilde{x}_{\infty}$ in $B(o,R)$ for any $R\geq R_0$ by Langer's result \cite{J}. Moreover, we have
$\langle\widetilde{x}_{\infty},\widetilde{\nu}_{\infty}\rangle\leq 0$ and $\widetilde{H}_{\infty}+\langle\widetilde{x}_{\infty},\widetilde{\nu}_{\infty}\rangle\geq 0$. By the (\ref{monotocity_formula1}),
we have $(\langle\widetilde{x}_{\infty},\widetilde{\nu}_{\infty}\rangle-\widetilde{H}_{\infty})(\langle\widetilde{x}_{\infty},\widetilde{\nu}_{\infty}\rangle+\widetilde{H}_{\infty})=0$.
Hence $\langle\widetilde{x}_{\infty},\widetilde{\nu}_{\infty}\rangle+\widetilde{H}_{\infty}=0$.
\end{rem}

\begin{proof}[Proof of Theorem \ref{asymptotically_app3}]
By rescaling the flow as in the proof of Theorem \ref{asymptotically_app2} allows us to assume without losing the generality that $x_0$ satisfies  $H+\langle x_0,\nu\rangle \leq 0$. By Lemma \ref{lemma_4.1}, $\widetilde{H}\geq 0$, $\langle\widetilde{x},\nu\rangle \le 0$ and $\widetilde{H}+\langle\widetilde{x},\widetilde{\nu}\rangle\leq 0$ for all $s\geq 0$.

By (\ref{monotocity_formula1}), we have
\begin{equation*}\label{monotocity_formula2}
\frac{\partial}{\partial s}e^{-\frac{1}{2}|\widetilde{x}|^2}d\widetilde{\mu}_s=\langle\widetilde{x},\widetilde{\nu}\rangle^2-\widetilde{H}^2=
(\langle\widetilde{x},\widetilde{\nu}\rangle-\widetilde{H})(\langle\widetilde{x},\widetilde{\nu}\rangle+\widetilde{H})\geq 0.
\end{equation*}
Then
\begin{equation*}
\frac{d}{d s}\int_Me^{-\frac{1}{2}|\widetilde{x}|^2}d\widetilde{\mu}_s=
\int_M(\langle\widetilde{x},\widetilde{\nu}\rangle-\widetilde{H})(\langle\widetilde{x},\widetilde{\nu}\rangle+\widetilde{H})e^{-\frac{1}{2}|\widetilde{x}|^2}d\widetilde{\mu}_s\geq 0.
\end{equation*}
and hence
\begin{equation*}
\int^{\infty}_0\int_M(\langle\widetilde{x},\widetilde{\nu}\rangle-\widetilde{H})
(\langle\widetilde{x},\widetilde{\nu}\rangle+\widetilde{H})e^{-\frac{1}{2}|\widetilde{x}|^2}d\widetilde{\mu}_s
\leq \lim\limits_{s\to\infty}\int_Me^{-\frac{1}{2}|\widetilde{x}|^2}d\widetilde{\mu}_s<\infty.
\end{equation*}
From this point on we can argue as in the proof of Theorem \ref{asymptotically_app2} given in Remark \ref{another-proof} to conclude that the normalized drifting mean curvature flow subconverges smoothly to the limiting self-expander.

Moreover,  if we have $H^n(\widetilde{x}(M,s)\cap B(o,R))\leq CR^m$, then
\begin{align*}
&\int_{M}e^{-\frac{1}{2}|\widetilde{x}|^2}d\widetilde{\mu}_s
=\sum\limits_{j=1}^{\infty}\int_{\widetilde{x}(M,s)\cap B(o,R^j)\setminus\widetilde{x}(M,s)\cap B(o,R^{j-1})}e^{-\frac{1}{2}|\widetilde{x}|^2}d\widetilde{\mu}_s\\
\leq& \sum\limits_{j=0}^{\infty}H^n(\widetilde{x}(M,s)\cap B(o,R^j)\setminus \widetilde{x}(M,s)\cap B(o,R^{j-1}))e^{-\frac{1}{2}R^{2j-2}}\\
\leq& \sum\limits_{j=0}^{\infty}CR^{mj}e^{-\frac{1}{2}R^{2j-2}}<\infty.
\end{align*}
\end{proof}

The example \ref{counterexample2} illustrates that the asymptotic limit of Type III mean curvature flow which is an entire graph with  positive mean curvature may not always be an expander. Corollary \ref{cor_app} describes under which additional condition on an initial hypersurface we can guarantee that the asymptotic limit in the situation described as above is always an expander.

\begin{proof}[Proof of Corollary \ref{cor_app}]
By the result of Ecker and Huisken \cite{EH}(see Proposition 4.4 in \cite{EH}), we know that the mean curvature flow of an entire graph satisfying the linear growth condition
\begin{equation}
\langle \nu,w\rangle^{-1}\leq c,
\end{equation}
  must be Type III, where $w$ is a fixed vector such that $\langle \nu,w\rangle>0$.
If $\langle x_0-q_0 ,\nu\rangle \leq C$ for some fixed vector $q_0$, then
we have $\langle x_0-q_0 -c_1w,\nu\rangle \leq C-c^{-1}c_1\leq 0$ for $c_1$ large enough.
By Corollary 3.2 in \cite{EH}, 
$\langle \nu,w\rangle^{-1}\leq c$ remains valid under the mean curvature flow.
Since $\nu$ is scaling invariant, we conclude that  $\langle \widetilde{\nu},w\rangle^{-1}\leq c$ remains valid under the normalized drifting mean curvature flow and $H^n({\widetilde{x}(M,s)\cap B(o,R)})= \int_{|x|\leq R}\langle \widetilde{\nu},w\rangle^{-1} dx^n\leq cR^n$.  By Theorem \ref{asymptotically_app3}, we know that its corresponding normalized drifting mean curvature flow (\ref{normalized_tangent_mcf}) subconverges smoothly to the limiting  self-expander.
\end{proof}

\thanks{\textbf{Acknowledgement}: First author would like to thank Math Department of Rutgers University for their hospitality, and is grateful to Professor Xiaochun Rong for his constant support and encouragement.}

\end{document}